\documentclass[tpa,preprint]{imsart}

\RequirePackage[OT1]{fontenc}
\RequirePackage{amsthm,amsmath}
\usepackage{graphicx,float,latexsym,times}
\usepackage{amsfonts,amstext,amssymb,amsthm}
\usepackage{mathrsfs}
\usepackage{bbm}

\newcommand{\dec}{\mathsf{d}}
\newcommand{\x}{\hat{\mathsf{x}}}
\newcommand{\Li}{\mathsf{L}}
\newcommand{\T}{\mathsf{T}}
\newcommand{\cT}{\mathcal{T}}
\newcommand{\ind}[1]{\mathbbm{1}_{\{#1\}}}   
\newcommand{\indno}[1]{\mathbbm{1}_{#1}}   
\newcommand{\Exp}{\mathsf{E}}                  
\newcommand{\Pro}{\mathsf{P}}                  
\newcommand{\Hyp}{\mathsf{H}}                  
\newcommand{\cF}{\mathscr{F}}
\newcommand{\cG}{\mathscr{G}}
\newcommand{\cH}{\mathscr{H}}
\newcommand{\cY}{\mathscr{Y}}
\newcommand{\cN}{\mathscr{N}}
\newcommand{\cC}{\mathscr{C}}

\newtheorem{theorem}{Theorem}
\newtheorem{lemma}{Lemma}

\pubyear{2013}
\volume{0}
\issue{0}
\firstpage{1}
\lastpage{17}

\begin{document}
\begin{frontmatter}
\title{Sequential Joint Detection and Estimation}
\runtitle{SEQUENTIAL JOINT DETECTION AND ESTIMATION}

\author{Yasin Yilmaz\thanks{Y. Yilmaz and X. Wang are with Columbia University, NY, USA.}, George V.~Moustakides\thanks{G.V. Moustakides is with the University of Patras, Patras, Greece.}, and Xiaodong Wang$^*$}
\runauthor{Y. YILMAZ, G.V. MOUSTAKIDES, X. WANG}

\begin{abstract}
\textbf{Abstract}. We consider the problem of simultaneous detection and estimation under a sequential framework. In particular we are interested in sequential tests that distinguish between the null and the alternative hypothesis and every time the decision is in favor of the alternative they provide an estimate of a random parameter. As we demonstrate with our analysis treating the two subproblems separately with the corresponding optimal strategies does not result in the best possible performance. To enjoy optimality one needs to take into account the optimum estimator during the hypothesis testing phase.
\end{abstract}

\end{frontmatter}

\section{Introduction} \label{sec:intro}

Suppose we are observing sequentially two processes $\{y_t\}$, $\{h_t\}$ which are related through the following model
\begin{equation}
y_t=xh_t+w_t;~~t=1,2,\ldots.
\label{eq:data_model}
\end{equation}
Process $\{w_t\}$ is a noise sequence; $x$ a random variable described by the following two hypotheses\\
$~~~~~\Hyp_0:~~x=0$, \\
$~~~~~\Hyp_1:~~x\sim\varphi(x),$\\
where $x\sim\varphi(x)$ means that the random variable $x$ follows the pdf $\varphi(x)$; and $\{h_t\}$ a second observed process that affects in a time-varying and random way the value of the random variable $x$. In other words, under the null hypothesis the observed sequence $\{y_t\}$ is pure noise whereas under the alternative hypothesis it contains a mean which is related to the random parameter $x$ and scaled through the second measured sequence $\{h_t\}$. 

Sequences of this form arise in several applications in practice, the most notable being digital communications where $x$ denotes the information to be transmitted. Under hypothesis $\Hyp_0$ no transmission takes place, consequently the receiver measures pure noise. Under hypothesis $\Hyp_1$ information $x$ is transmitted and the sequence $\{h_t\}$ models the attenuation inflicted on this variable by a lossy and time-varying communication channel. We should mention that in digital communications it is customary to consider that the channel sequence $\{h_t\}$ can be measured, consequently, assuming that this process is available, is realistic (see Proakis and Salehi (2008)).

The mathematical challenge we would like to consider in this work consist: a)~in \textit{deciding as soon as possible} between the two hypotheses, and b)~every time a decision is made in favor of $\Hyp_1$, to \textit{provide an estimate of the random variable $x$}. As we realize, we have a \textit{joint} detection and estimation problem where both subproblems are of equal importance. Indeed we note that \textit{we like to have a reliable estimate} of $x$ every time \textit{we detect its presence}.

Key element in our formulation constitutes the fact that we are interested in performing the joint detection/estimation task as soon as possible, suggesting that we intend to focus on \textit{sequential} schemes to solve the joint problem. Finally, we would like to emphasize that our analysis is going to demonstrate that solving the joint problem by treating each subproblem separately with the corresponding optimal procedure does not yield an overall optimum performance. As we shall see, the detection part needs to take into account the fact that we are also interested in parameter estimation in order for the combined scheme to perform optimally.

Sequential joint detection/estimation differs from sequential composite hypothesis testing where parameters are either marginalized or treated as nuisance (see Lerche (1986), Pavlov (1990)). Actually, joint detection/estimation resembles to sequential multi-hypothesis testing where there is a \textit{discrete} set of possible probability measures that describe the observations and we need to select one of the existing possibilities. Characteristic articles treating this problem are: Armitage (1950), Lorden (1977), Tartakovsky (1998) and Dragalin et al. (1999). The joint case studied in this work differs from the previous approaches in the sense that we have a parametric family of measures (parametrized by $x$) and we need to select the correct parameter value, \textit{after} establishing that this value is not 0. Existing literature related to joint detection/estimation is very limited and addressing only the fixed sample size case. The articles by Middleton and Esposito (1968), Fredriksen et al.~(1972), Moustakides (2011) and Moustakides et al.~(2012), offer different formulation possibilities for the fixed sample size version. In the current work, we are focusing on the setup proposed by Moustakides (2011) and extend the corresponding result to the sequential case.

Let us now become more technical by introducing the detection/estimation strategies we are interested in. Assuming observations become available sequentially in pairs $\{(y_t,h_t)\}$, let $\{\cF_t\}_{t\ge0}$ denote the corresponding filtration with $\cF_t=\sigma\{(y_1,h_1),\ldots,(y_t,h_t)\}$ and $\cF_0$ the trivial $\sigma$-algebra. We also define two additional filtrations $\{\cY_t\}_{t\ge0}$ and $\{\cH_t\}_{t\ge0}$ with $\cY_t=\sigma\{y_1,\ldots,y_t\}$ and $\cH_t=\sigma\{h_1,\ldots,h_t\}$, that is, the accumulated history pertinent to the first and second observed sequence respectively, and $\cY_0,\cH_0$ being, again,  trivial $\sigma$-algebras. We clearly have $\cF_t=\cY_t\cup\cH_t$, therefore $\cH_t\subseteq\cF_t$. 

According to what we mentioned, we are looking for a triplet $(T,d_T,\hat{x}_T)$ where $T$ is a stopping time, $d_T$ a decision rule that distinguishes between the two hypotheses and $\hat{x}_T$ an estimator for $x$. The detector $d_T$ and the estimator $\hat{x}_T$ are $\cF_T$-measurable functions, namely they use all available information acquired up to time of stopping $T$, for deciding between the two hypotheses and for providing an estimate for $x$ every time this is deemed necessary (i.e.~whenever the detector decides in favor of $\Hyp_1$). For the stopping time $T$, the obvious choice would be to ask this quantity to be $\{\cF_t\}$-adapted, namely, at each time $t$ to use all available information to decide whether to stop or continue sampling. Unfortunately, imposing this requirement induces serious analytical complications. This fact is already known for the two \textit{separate} subproblems of detection and estimation. For instance if we assume that we always have $y_t=xh_t+w_t$ and we are interested in estimating $x$ then, as it is mentioned in Ghosh (1987) and in Ghosh and Sen (1991), finding the optimum sequential estimator of $x$ is not a tractable problem if $T$ is adapted to the complete observation history $\{\cF_t\}$. Instead, Grambsch (1983) and more recently Fellouris (2012), proposed to limit $T$ to $\{\cH_t\}$-adapted strategies, assumption that leads to simple and interesting optimal solution.

Similar analytical difficulties arise in the pure sequential hypothesis testing problem of distinguishing between $\Hyp_0$ and $\Hyp_1$. If we require $T$ to be $\{\cF_t\}$-adapted and attempt to solve this problem following, for example, the classical approach of Wald and Wolfowitch (1948), then the optimum scheme is not the usual SPRT as one would expect. This is because by observing the pair process $\{(y_t,h_t)\}$ we end up with a two-dimensional optimal stopping problem which is impossible to solve (analytically) since the thresholds for the running likelihood ratio will depend on the sequence $\{h_t\}$. Only if the sequence $\{h_t\}$ is constant, or not observed (or even ignored) and, additionally, we assume it is i.i.d.~with known pdf, then the detection problem can be reduced to the one considered by Wald and Wolfowitz (1948), accepting as solution the classical SPRT. In this case the stopping time $T$ becomes $\{\cY_t\}$-adapted and the decision function $d_T$ must be selected to be $\cY_T$-measurable. 

An alternative direction would be to consider, as in the pure estimation problem, $\{\cH_t\}$-adapted stopping times but, as we suggested above, allow the decision function $d_T$ to have access to the complete information, that is, be $\cF_T$-measurable. This is the approach we intend to adopt in this work. In fact we are going to apply this idea directly to the more general joint detection/estimation problem. As we shall see, our analysis will also offer the solution to the pure detection problem by proper parameter selection. Next we summarize our assumptions.

{\sc Assumptions:} i)~The two processes $\{w_t\},\{h_t\}$ are independent and independent from the random variable $x$ with the noise process $\{w_t\}$ being i.i.d.~with $w_t\sim\cN(0,\sigma^2)$, where $\cN(a,b^2)$ denotes Gaussian pdf with mean $a$ and variance $b^2$. ii)~For $x$, under $\Hyp_1$, we assume that $x\sim\cN(\mu_x,\sigma_x^2)$, in other words the prior $\varphi(x)$ is the Gaussian pdf; while under $\Hyp_0$ we assume $x=0$. Parameters $\mu_x,\sigma_x,\sigma$ are considered known. iii)~For the second observation process $\{h_t\}$ we only make the very mild assumption
\begin{equation}
\Pro\left(\sum_{t=0}^\infty h_t^2=\infty\right)=1,
\label{eq:wp1}
\end{equation}
that is, with probability 1, each realization of this process has infinite energy over the infinite time horizon. No other condition is imposed on $\{h_t\}$, consequently, for this process no \textit{dependency or time variability model is specified}, and the actual statistical description \textit{is not required to be known}. iv)~The stopping time $T$ is $\{\cH_t\}$-adapted while the decision function $d_T$ and the estimator $\hat{x}_T$ are $\cF_T$-measurable functions having access to the complete observation history up to the time of stopping.

In the rest of our article, with $\Pro_0,\Exp_0$ we denote probability measure and expectation under hypothesis $\Hyp_0$; with $\bar{\Pro}_1,\bar{\Exp}_1$ probability measure and expectation under hypothesis $\Hyp_1$ including the statistical description of the random variable $x$ and, finally, with $\Pro_1,\Exp_1$ probability measure and expectation under hypothesis $\Hyp_1$ but with $x$ being marginalized.

Before continuing with our problem formulation it would be worth mentioning a very recent article by Cetin, Novikov and Shiryaev (2013) that refers to the pure parameter estimation problem, treating a very similar data model as the one introduced in \eqref{eq:data_model}. The basic difference between the two data types is that in our case, as we have pointed out in our assumptions, the two sequences $\{h_t\},\{w_t\}$ are in discrete time and are \textit{independent}; whereas in Cetin, Novikov and Shiryaev (2013) they are continuous-time processes \textit{related} through a diffusion type stochastic differential equation. This difference allows for completely different mathematical setups, even though the final optimum procedures turn out to be very similar.

\section{Problem formulation}
In sequential detection and estimation we are usually interested in minimizing the average delay subject to suitable constraints. However, in order to free our formulation from the need to specify a probability measure for the process $\{h_t\}$, we will adopt the same idea employed in sequential estimation, namely consider expected delays, error probabilities and average costs \textit{conditioned} on the sequence $\{h_t\}$. This approach will give rise to a triplet $(T,d_T,\hat{x}_T)$ which will be optimum for \textit{each realization} of $\{h_t\}$ and not on average with respect to this sequence, as is the usual case in classical Sequential Analysis.

Since we are interested in the two subproblems of detection and estimation we have a number of quantities that are pertinent to each case. For the detection part we have the Type-I and Type-II error probabilities that accept the following conditional form: $\Pro_0(d_T=1|\cH_T)$ and $\Pro_1(d_T=0|\cH_T)$. For the estimation problem we assume that we are under hypothesis $\Hyp_1$ and we adopt as cost function the mean squared error. We recall that our estimate depends on the decision of our detector, in particular: whenever $d_T=1$ we provide an estimate $\hat{x}_T$ which inflicts a squared estimation error $(\hat{x}_T-x)^2$, where $x$ is the true value of our random parameter. Alternatively, when the detector erroneously decides in favor of $\Hyp_0$, that is, $d_T=0$, then this is like estimating our parameter as $\hat{x}_T=0$ (since under $\Hyp_0$ we have $x=0$) generating a squared error $(0-x)^2=x^2$. Consequently, for the estimation subproblem there are the following two conditional mean squared errors that are of interest:
$\bar{\Exp}_1[(\hat{x}_T-x)^2\ind{d_T=1}|\cH_T]$ and $\bar{\Exp}_1[x^2\ind{d_T=0}|\cH_T]$, where $\indno{A}$ denotes the indicator of the event $A$. 

Now, we can use these four quantities to form the following \textit{combined cost function}
\begin{multline}
\cC(T,d_T,\hat{x}_T)=c_0\Pro_0(d_T=1|\cH_T)+c_1\Pro_1(d_T=0|\cH_T)\\
+c_e\bar{\Exp}_1\left[(\hat{x}_T-x)^2\ind{d_T=1}+x^2\ind{d_T=0}|\cH_T\right]
\label{eq:eq2}
\end{multline}
where $c_0,c_1,c_e$ are nonnegative values selected by the Statistician. The last term in the right hand side of \eqref{eq:eq2}, which refers to the estimation problem, as we can see, depends on both parts, namely our decision and our estimation strategy. Furthermore, we note that if we set $c_e=0$ then the combined cost depends only on the decision rule $d_T$ suggesting that our joint problem is reduced into a pure detection problem.

To define an optimum joint scheme, we will follow a constrained optimization approach, therefore we are going to consider triplets $(T,d_T,\hat{x}_T)$ for which the combined cost $\cC(T,d_T,\hat{x}_T)$ is upper bounded by some prescribed value. From the class of triplet strategies which is generated through this constraint we will select the one that \textit{minimizes} the stopping time $T$. More specifically we would like to solve the following constrained optimization problem:
\begin{equation}
\inf_{T,d_T,\hat{x}_T}T;~~~~~\text{subject to:}~\cC(T,d_T,\hat{x}_T)\leq C,
\label{eq:eq3}
\end{equation}
where $C>0$ is the maximal combined cost we are willing to tolerate. Note that since $T$ is $\{\cH_t\}$-adapted, as we mentioned before and would like to emphasize once more, the triplet we are going to develop will minimize $T$ for \textit{each} realization of the process $\{h_t\}$ and not $\Exp[T]$, where the average is taken over \textit{all} realizations of this process, as is the usual case in classical optimal stopping problems.

\section{Optimum solution}
The optimum triplet will be obtained in three steps. First we will propose a candidate estimator by solving a smaller auxiliary optimization problem, then we are going to use this solution to propose a candidate detector that takes into account the previous estimator by solving a second auxiliary optimization problem and, in the end, we will provide a candidate stopping time and show that all three proposed parts constitute the triplet that solves the original constrained optimization problem depicted in \eqref{eq:eq3}. Let us continue by first identifying our candidate estimator.

\subsection{Optimum estimation}
Fix the stopping time $T$ assuming that it is finite with probability 1 and the decision function $d_T$. Consider the problem of minimizing the conditional mean squared error $\bar{\Exp}_1[(\hat{x}_T-x)^2\ind{d_T=1}|\cH_T]$ with respect to the estimator $\hat{x}_T$. We have the following lemma that gives the solution to this problem and also provides a useful expression for the second term $\bar{\Exp}_1[x^2\ind{d_T=0}|\cH_T]$ of the estimation cost.

\begin{lemma}\label{lem:1}
The optimum estimator $\x_T$ that minimizes the conditional mean squared error $\bar{\Exp}_1[(\hat{x}_T-x)^2\ind{d_T=1}|\cH_T]$ with respect to $\hat{x}_T$, on the event $\{T=t\}$, is given by the following formula
\begin{equation}
		\x_t = \frac{V_t + \mu_x\frac{\sigma^2}{\sigma_x^2}}{U_t + \frac{\sigma^2}{\sigma_x^2}};~~\text{where}~~V_t=\sum_{n=1}^t y_n h_n;~~~U_t=\sum_{n=1}^t h_n^2,
\label{eq:est}
\end{equation}
while the corresponding minimum value of the conditional mean squared error takes the form
\begin{equation}
\inf_{\hat{x}_T} \bar{\Exp}_1[(\hat{x}_T-x)^2\ind{d_T=1}|\cH_T]=\frac{\sigma^2}{U_T + \frac{\sigma^2}{\sigma_x^2}}\Pro_1\left(d_T=1|\cH_T\right).
\label{eq:mmse1}
\end{equation}
Additionally we can write
\begin{equation}
\bar{\Exp}_1\left[x^2\ind{d_T=0}|\cH_T\right]=\Exp_1\left[
\x_T^2\ind{d_T=0} |\cH_T\right]
+\frac{\sigma^2}{U_T + \frac{\sigma^2}{\sigma_x^2}}\Pro_1\left(d_T=0|\cH_T\right).
\label{eq:mmse2}
\end{equation}
\end{lemma}
\begin{proof}
The proof is simple and based on the well known result that the mean squared error is minimized by the conditional mean of $x$ given all available observation history. The interesting detail is that this result is still valid even if the observation history is dictated by an $\{\cH_t\}$-adapted stopping time $T$ and an $\cF_T$-measurable decision rule $d_T$. To demonstrate \eqref{eq:mmse1}, using that $T$ is $\{\cH_t\}$-adapted, $d_T$ is $\cF_T$-measurable and $\cH_t\subseteq\cF_t$, we can write
\begin{multline*}
\bar{\Exp}_1[(\hat{x}_T-x)^2\ind{d_T=1}|\cH_T]=\bar{\Exp}_1\left[\sum_{t=0}^\infty(\hat{x}_t-x)^2\ind{d_t=1}\ind{T=t}\Big|\cH_t\right]\\
=\sum_{t=0}^\infty\bar{\Exp}_1\left[(\hat{x}_t-x)^2\ind{d_t=1}|\cH_t\right]\ind{T=t}\\
=\sum_{t=0}^\infty\Exp_1\left[\bar{\Exp}_1\left[(\hat{x}_t-x)^2|\cF_t\right]\ind{d_t=1}|\cH_t\right]\ind{T=t},
\end{multline*}
where for the last equality we used the tower property of expectation. From classical estimation theory (e.g.~Poor 1994, page 151) we know that 
$$
\inf_{\hat{x}_t}\bar{\Exp}_1\left[(\hat{x}_t-x)^2|\cF_t\right]=\frac{\sigma^2}{U_t+\frac{\sigma^2}{\sigma_x^2}}.
$$
This minimal value is attained by the conditional expectation $\x_t=\bar{\Exp}[x|\cF_t]$ which, due to the fact that $x$, given $\cF_t$, is Gaussian with mean $(V_t+\mu_x\frac{\sigma^2}{\sigma_x^2})/(U_t+\frac{\sigma^2}{\sigma_x^2})$ and variance $\sigma^2/(U_t + \frac{\sigma^2}{\sigma_x^2})$, is equal to
$$
\x_t=\frac{V_t + \mu_x\frac{\sigma^2}{\sigma_x^2}}{U_t + \frac{\sigma^2}{\sigma_x^2}}.
$$
Consequently, since $U_t$ is $\cH_t$-measurable we deduce
\begin{multline*}
\bar{\Exp}_1[(\hat{x}_T-x)^2\ind{d_T=1}|\cH_T]\geq \sum_{t=0}^\infty\Exp_1\left[\frac{\sigma^2}{U_t+\frac{\sigma^2}{\sigma_x^2}}\ind{d_t=1}|\cH_t\right]\ind{T=t}\\
=\sum_{t=0}^\infty\frac{\sigma^2}{U_t+\frac{\sigma^2}{\sigma_x^2}}\Pro_1(d_t=1|\cH_t)\ind{T=t}
=\frac{\sigma^2}{U_T+\frac{\sigma^2}{\sigma_x^2}}\Pro_1(d_T=1|\cH_T),
\end{multline*}
with equality, as we mentioned, when $\hat{x}_t=\x_t$ on $\{T=t\}$.

To prove \eqref{eq:mmse2}, we can write
\begin{multline}
\bar{\Exp}_1\left[x^2\ind{d_T=0}|\cH_T\right]=\sum_{t=0}^\infty\bar{\Exp}_1\left[x^2\ind{d_t=0}|\cH_t\right]\ind{T=t}\\
=\sum_{t=0}^\infty\Exp_1\left[\bar{\Exp}_1\left[x^2|\cF_t\right]\ind{d_t=0}|\cH_t\right]\ind{T=t}.
\label{eq:mbafla1}
\end{multline}
Using again, as we mentioned above, the fact that $x$ conditioned on $\cF_t$ is Gaussian with mean $\x_t$ and variance $\sigma^2/(U_t + \sigma^2/\sigma_x^2)$, we compute
$$
\bar{\Exp}_1\left[x^2|\cF_t\right]=\x_t^2+\frac{\sigma^2}{U_t + \frac{\sigma^2}{\sigma_x^2}}.
$$
Substituting this equality in \eqref{eq:mbafla1} and recalling that $U_t$ is $\cH_t$-measurable, yields the desired result.
\end{proof}

\subsection{Optimum detection}
If we consider the combined cost $\cC(T,d_T,\x_T)$ where the estimator $\hat{x}_T$ is replaced by the optimum $\x_T$ defined in \eqref{eq:est} then, using \eqref{eq:mmse1} and \eqref{eq:mmse2} we obtain
\begin{multline*}
\cC(T,d_T,\x_T)=c_0\Pro_0(d_T=1|\cH_T)+c_1\Pro_1(d_T=0|\cH_T)\\
+c_e\Exp_1\left[
\x_T^2\ind{d_T=0} |\cH_T\right]+c_e\frac{\sigma^2}{U_T + \frac{\sigma^2}{\sigma_x^2}}.
\end{multline*}
Due to the fact that $\x_T$ is the result of the minimization stated in Lemma\,\ref{lem:1} we have
$\cC(T,d_T,\x_T)\leq \cC(T,d_T,\hat{x}_T)$.
We note that the last term in the expression for $\cC(T,d_T,\x_T)$ does not depend on the decision function $d_T$, therefore, let us consider the sum of the first three terms of the right hand side and define the auxiliary cost
\begin{multline}
\tilde{\cC}(T,d_T)=c_0\Pro_0(d_T=1|\cH_T)+c_1\Pro_1(d_T=0|\cH_T)\\
+c_e\Exp_1\left[
\x_T^2\ind{d_T=0} |\cH_T\right].
\label{eq:Caux}
\end{multline}
In the sequel our goal is, for fixed $T$, to identify the decision function $\dec_T$ that minimizes $\tilde{\cC}(T,d_T)$ with respect to $d_T$. The solution to this problem is given in the next lemma.

\begin{lemma}\label{lem:2}
The decision function $\dec_T$ that minimizes the auxiliary cost function $\tilde{\cC}(T,d_T)$ with respect to $d_T$, on the event $\{T=t\}$, is given by the following formula
\begin{equation}
\dec_t=\left\{
\begin{array}{cl}
1&\text{if}~~c_0\leq \Li_t\left\{c_1+c_e\x_t^2
\right\}\\
0&\text{otherwise},
\end{array}
\right.
\label{eq:lem2.1}
\end{equation}
where $\Li_t$ is the conditional likelihood ratio of the pdfs of the two hypotheses given $\cH_t$, with the random variable $x$ under $\Hyp_1$ being marginalized, specifically
\begin{equation}
\Li_t=\frac{1}{\sqrt{U_t+\frac{\sigma^2}{\sigma_x^2}}}\frac{\sigma}{\sigma_x}e^{\frac{\left(V_t+\mu_x\frac{\sigma^2}{\sigma_x^2}\right)^2}{2\sigma^2\left(U_t+\frac{\sigma^2}{\sigma_x^2}\right)}-\mu_x^2\frac{1}{2\sigma_x^2}}.
\label{eq:lem2.2}
\end{equation}
The resulting minimum value of the auxiliary cost function takes the form
\begin{multline}
\inf_{d_T}\tilde{\cC}(T,d_T)
=\Exp_0\left[
\left(c_0-\Li_T\left\{c_1+c_e\x_T^2
\right\}\right)^{\!\!-} |\cH_T\right]\\
+c_1+c_e\left\{
\mu_x^2+\frac{\sigma_x^2U_T}{U_T+\frac{\sigma^2}{\sigma_x^2}}
\right\},
\label{eq:lem2.3}
\end{multline}
where $z^-=\min\{z,0\}$.
\end{lemma}

\begin{proof}
The proof of this lemma presents no special difficulty. We can write
\begin{align}
\Pro_0(d_T=1|\cH_T)&=\sum_{t=0}^\infty\Exp_0[\ind{d_t=1}|\cH_t]\ind{T=t}\label{eq:mbofla1}\\
\Pro_1(d_T=0|\cH_T)&=\sum_{t=0}^\infty\Exp_0[\Li_t\ind{d_t=0}|\cH_t]\ind{T=t}.\label{eq:mbofla2}
\end{align}
Similarly we have
\begin{equation}
\Exp_1\left[
\x_T^2\ind{d_T=0} |\cH_T\right]
=\sum_{t=0}^\infty\Exp_0\left[\Li_t\x_t^2\ind{d_t=0} |\cH_t\right]\ind{T=t},
\label{eq:mbofla3}
\end{equation}
where we used the fact that $\x_t$ is $\cF_t$-measurable and $\Li_t$ is the corresponding conditional likelihood ratio of the two hypotheses. Substituting \eqref{eq:mbofla1},\eqref{eq:mbofla2},\eqref{eq:mbofla3}, in the definition of the auxiliary cost $\tilde{\cC}(T,d_T)$ in \eqref{eq:Caux} we obtain
\begin{multline*}
\tilde{\cC}(T,d_T)=\sum_{t=0}^\infty\Exp_0\left[
c_0\ind{d_t=1}+\Li_t\left\{c_1+c_e\x_t^2\right\}\ind{d_t=0}
|\cH_t\right]\ind{T=t}\\
=\sum_{t=0}^\infty\Exp_0\left[
\left(c_0-\Li_t\left\{c_1+c_e\x_t^2\right\}\right)\ind{d_t=1} |\cH_t\right]\ind{T=t}\\
+\sum_{t=0}^\infty\Exp_0\left[\Li_t\left\{c_1+c_e\x_t^2\right\} |\cH_t\right]\ind{T=t}\\
\geq\sum_{t=0}^\infty\Exp_0\left[
\left(c_0-\Li_t\left\{c_1+c_e\x_t^2\right\}\right)^-
 |\cH_t\right]\ind{T=t}\\
+\sum_{t=0}^\infty\Exp_1\left[
c_1+c_e\x_t^2
 |\cH_t\right]\ind{T=t}.
\end{multline*}
We can easily verify that we have equality when the decision function is according to \eqref{eq:lem2.1}. In the  last sum in the previous expression it can be shown that the corresponding expectation is equal to $c_1+c_e\{
\mu_x^2+\sigma_x^2U_T/(U_T+\frac{\sigma^2}{\sigma_x^2})\}$. Indeed this is true because $V_T$ on the event $\{T=t\}$, under $\Hyp_1$ and conditioned on $\cH_t$, is Gaussian with mean $\mu_xU_t$ and variance $\sigma_x^2U_t^2+\sigma^2 U_t=\sigma_x^2U_t(U_t+\frac{\sigma^2}{\sigma_x^2})$.

To show the validity of \eqref{eq:lem2.2} we have that the likelihood ratio of the two hypotheses, given $x$ and $\cH_t$, is equal to $\exp(-\frac{x^2}{2\sigma^2}U_t+\frac{x}{\sigma^2}V_t)$. Marginalizing $x$ using the Gaussian prior yields $\Li_t$ which can therefore be computed as
$$
\Li_t=\int_{-\infty}^{\infty} e^{-\frac{x^2}{2\sigma^2}U_t+\frac{x}{\sigma^2}V_t}\frac{1}{\sqrt{2\pi\sigma_x^2}}e^{-\frac{1}{2\sigma_x^2}(x-\mu_x)^2}\,dx.
$$
Combining the two exponents and ``completing the square'' for $x$, it is straightforward to prove \eqref{eq:lem2.2}.
\end{proof}

From \eqref{eq:lem2.1} if we set $c_e=0$, we end up with the pure detection problem, and the optimum detector reduces to the usual likelihood ratio test which is applied at the time of stopping $T$. However, when $c_e>0$, in the detection rule we take into account the \textit{optimum estimate}, and our detector is no longer a likelihood ratio test. Actually, this is exactly the point that discriminates our optimum joint detection/estimation scheme from the approach that solves the two problems separately by applying the corresponding optimum strategies. Note that the latter method would have simply applied the likelihood ratio test for detection and then the optimum estimator whenever the decision was in favor of $\Hyp_1$. Our scheme on the other hand makes a decision by taking into account the square of the optimum estimate.

\subsection{Optimum stopping time}
Using the results of Lemma\,\ref{lem:2}, in particular substituting \eqref{eq:lem2.3} in the combined cost function, we obtain
\begin{multline}
\cC(T,\dec_T,\x_T)=\Exp_0\left[
\left(c_0-\Li_T\left\{c_1+c_e\x_T^{2}
\right\}\right)^{-} |\cH_T\right]+c_1+c_e(\mu_x^2+\sigma_x^2)
\label{eq:opt2}
\end{multline}
From the way $\dec_T,\x_T$ were defined, we clearly deduce that any triplet $(T,d_T,\hat{x}_T)$ satisfies the following inequality
\begin{equation}
\cC(T,\dec_T,\x_T)\leq\cC(T,d_T,\hat{x}_T).
\label{eq:ineq2}
\end{equation}
Let us now make a more explicit computation of the conditional expectation appearing in \eqref{eq:opt2}. For this reason, in the next lemma we define a suitable function $\cG(U)$ for which we also prove a monotonicity property that plays a crucial role in specifying the final term of our desired triplet, namely the optimum stopping time. The lemma is based on the observation that on the event $\{T=t\}$ and given $\cH_t$, we have $U_t$ known and, under $\Hyp_0$, $V_t\sim\cN(0,\sigma^2U_t)$. 

\begin{lemma}\label{lem:3}
For $U\geq0$, define the following function
\begin{multline}
\cG(U)=\\
\int_{-\infty}^{\infty}\!\!\left(c_0-
\frac{\frac{\sigma}{\sigma_x}e^{\frac{\left(V+\mu_x\frac{\sigma^2}{\sigma_x^2}\right)^2}{2\sigma^2\left(U+\frac{\sigma^2}{\sigma_x^2}\right)}-\mu_x^2\frac{1}{2\sigma_x^2}}}{\sqrt{U+\frac{\sigma^2}{\sigma_x^2}}}
\left[c_1+c_e\left(\frac{V+\mu_x\frac{\sigma^2}{\sigma_x^2}}{U+\frac{\sigma^2}{\sigma_x^2}}\right)^{\!\!2}
\right]\right)^{\!\!-}\!\!\!\frac{e^{-\frac{1}{2\sigma^2U}V^2}}{\sqrt{2\pi\sigma^2 U}}\,dV,
\label{eq:GU}
\end{multline}
then $\cG(U)$ is continuous, strictly decreasing in $U\geq0$, with $\lim_{U\to\infty}\cG(U)=-c_1-c_e(\mu_x^2+\sigma_x^2)$ and $\cG(0)=(c_0-c_1-c_e\mu_x^2)^-$.
\end{lemma}

\begin{proof}
Because the proof is very technical, we will not present all computational details. That $\cG(U)$ is continuous it is obvious since the integrand is continuous in $U$ and $V$. Let us now prove the desired monotonicity property of $\cG(U)$. For simplicity, call $\kappa=\frac{\sigma^2}{\sigma_x^2}$ and define the function $G(U,V)$
\begin{multline}
G(U,V)=\\
\left(c_0-
\sqrt{\frac{\kappa}{U+\kappa}}e^{\frac{(V+\mu_x\kappa)^2}{2\sigma^2(U+\kappa)}-\frac{\mu_x^2\kappa}{2\sigma^2}}
\left[c_1+c_e\left(\frac{V+\mu_x\kappa}{U+\kappa}\right)^2
\right]\right)\frac{e^{-\frac{1}{2\sigma^2U}V^2}}{\sqrt{2\pi\sigma^2 U}}\\
=c_0\frac{e^{-\frac{1}{2\sigma^2U}V^2}}{\sqrt{2\pi\sigma^2 U}}
-\left[c_1+c_e\left(\frac{V+\mu_x\kappa}{U+\kappa}\right)^2\right]
\frac{e^{-\frac{(V-\mu_xU)^2}{2\sigma_x^2U(U+\kappa)}}}{\sqrt{2\pi\sigma_x^2U(U+\kappa)}}
\label{eq:lem3.0}
\end{multline}
Denote with $g(U)$ the solution of the equation
\begin{equation}
c_0=\sqrt{\frac{\kappa}{U+\kappa}}e^{\frac{g}{2\sigma^2(U+\kappa)}-\frac{\mu_x^2\kappa}{2\sigma^2}}
\left[c_1+c_e\frac{g}{(U+\kappa)^2}\right],
\label{eq:lem3.1}
\end{equation}
where $g$ replaces $(V+\mu_x\kappa)^2$. Even though the latter quantity is nonnegative we allow $g$ to take also negative values thus guaranteeing that \eqref{eq:lem3.1} has always a solution. Indeed
note that the right hand side in \eqref{eq:lem3.1} is strictly decreasing in $U\geq0$ and strictly increasing in $g$. For fixed $U$ if we set $g=-(U+\kappa)^2c_1/c_e$, the right hand side becomes 0. On the other hand by letting $g\to\infty$, the right hand side tends to $\infty$ as well. Due to continuity and strict increase in $g$ there is a unique solution $g(U)$. 

Using $g(U)$ we can now deduce that the values of $V$ for which the integrand in \eqref{eq:GU} and therefore $G(U,V)$ is nonpositive is $V\in\cT(U)=(-\infty,-V_1(U)]\cup[V_2(U),\infty)$ where $V_1(U)=\sqrt{g^+(U)}+\mu_x\kappa$, $V_2(U)=\sqrt{g^+(U)}-\mu_x\kappa$, and $z^+=\max\{z,0\}$. Note that for values of $U$ for which $g(U)\leq0$ we have $-V_1(U)=V_2(U)=-\mu_x\kappa$, therefore both quantities coincide. When, however, $g(U)>0$ then $G(U,-V_1(U))=G(U,V_2(U))=0$.
Using the previous definitions and observations we have the following expressions for $\cG(U)$
\begin{multline}
\cG(U)=\int_{-\infty}^{-V_1(U)}G(U,V)dV+\int_{V_2(U)}^\infty G(U,V)dV\\
=\int_{\cT(U)}G(U,V)dV=\int_{-\infty}^\infty G(U,V)\indno{\cT(U)}(V)dV.
\label{eq:lem3.mbifla1}
\end{multline}

To show that $\cG(U)$ is decreasing, it suffices to show that its derivative is negative. 
Let first $U$ be such that the solution to \eqref{eq:lem3.1} satisfies $g(U)\leq0$. In this case, as we mentioned, we have $-V_1(U)=V_2(U)=-\mu_x\kappa$ suggesting that $\cT(U)$ becomes the whole real line. Thus, substituting \eqref{eq:lem3.0} in \eqref{eq:lem3.mbifla1}, we can write
\begin{multline*}
\cG'(U)=\left(\int_{-\infty}^{\infty}G(U,V)dV\right)'=\left(c_0-c_1-c_e\left\{\mu_x^2+\frac{\sigma_x^2 U}{U+\kappa}\right\}\right)'\\
=-c_e\frac{\sigma_x^2\kappa}{(U+\kappa)^2}<0,
\end{multline*}
and, therefore, $\cG(U)$ is strictly decreasing for all $U\geq0$ for which $g(U)\leq0$.

Let now $U$ be such that the solution to \eqref{eq:lem3.1} satisfies $g(U)>0$. Substituting again \eqref{eq:lem3.0} in \eqref{eq:lem3.mbifla1} and changing variables $z=V/\sqrt{U}$ we have
$$
\cG(U)=\int_{-\infty}^{-\bar{V}_1(U)}\bar{G}(U,z)dz+\int_{\bar{V}_2(U)}^{\infty}\bar{G}(U,z)dz=\int_{\bar{\cT}(U)}\bar{G}(U,z)dz
$$
where
\begin{multline*}
\bar{G}(U,z)=\sqrt{U}G(U,z\sqrt{U})\\
=c_0\frac{e^{-\frac{z^2}{2\sigma^2}}}{\sqrt{2\pi\sigma^2}}-\left[c_1+c_e\left(\frac{z\sqrt{U}+\mu_x\kappa}{U+\kappa}\right)^2\right]\frac{e^{-\frac{(z-\mu_x\sqrt{U})^2}{2\sigma_x^2(U+\kappa)}}}{\sqrt{2\pi\sigma_x^2(U+\kappa)}},
\end{multline*}
and $\bar{\cT}(U)=(-\infty,-\bar{V}_1(U)]\cup[\bar{V}_2(U),\infty)$ with $\bar{V}_i(U)=V_i(U)/\sqrt{U};~i=1,2$.
As before it is true that $\bar{G}(U,-\bar{V}_1(U))=\bar{G}(U,\bar{V}_2(U))=0$. Taking the derivative with respect to $U$ yields
\begin{multline*}
\cG'(U)=-\bar{G}(U,-\bar{V}_1(U))\bar{V}'_1(U)-\bar{G}(U,\bar{V}_2(U))\bar{V}'_2(U)+\int_{\bar{\cT}(U)}\partial_U\bar{G}(U,z)dz\\
=\int_{\bar{\cT}(U)}\partial_U\bar{G}(U,z)dz\\
=-\int_{\bar{\cT}(U)}\partial_U \left( \left[c_1+c_e\left(\frac{z\sqrt{U}+\mu_x\kappa}{U+\kappa}\right)^2\right]\frac{e^{-\frac{(z-\mu_x\sqrt{U})^2}{2\sigma_x^2(U+\kappa)}}}{\sqrt{2\pi\sigma_x^2(U+\kappa)}}\right) dz.
\end{multline*}
The latter integral after some tedious mathematical manipulations can be computed explicitly yielding
\begin{multline*}
\cG'(U)=-\frac{c_1\sqrt{g(U)}}{2(U+\kappa)\sqrt{2\pi\sigma_x^2U(U+\kappa)}}\,\Omega(U)\\
-\frac{c_e \sigma^4}{(U+\kappa)^2}\left\{\Phi\left(-\frac{\sqrt{g(U)}+\mu_x(U+\kappa)}{\sigma_x\sqrt{U(U+\kappa)}}\right) + \Phi\left(-\frac{\sqrt{g(U)}-\mu_x(U+\kappa)}{\sigma_x\sqrt{U(U+\kappa)}}\right) \right\}\\
-\frac{c_e \sigma^4\sqrt{g(U)}}{(U+\kappa)^2\sqrt{2\sigma_x^2\pi U(U+\kappa)}} \left( \frac{g(U)}{2\sigma_x^2 \kappa(U+\kappa)} +1 \right)\Omega(U).
\end{multline*}
where
$$
\Omega(U)=
e^{-\frac{[\sqrt{g(U)}+\mu_x(U+\kappa)]^2}{2\sigma_x^2U(U+\kappa)}}+e^{-\frac{[\sqrt{g(U)}-\mu_x(U+\kappa)]^2}{2\sigma_x^2U(U+\kappa)}},
$$
and $\Phi(x)$ is the standard Gaussian cdf.
We realize that all parts involving $c_1$ and $c_e$ are negative, suggesting that $\cG(U)$ is strictly decreasing. This is still true even if we limit ourselves to the pure detection problem by enforcing $c_e=0$.

To conclude our proof we need to show the validity of the formulas for $\cG(0)$ and $\lim_{U\to\infty}\cG(U)$. For $U\to0$ the term $(e^{-V^2/2\sigma^2U})/\sqrt{2\pi\sigma^2 U}$ in \eqref{eq:GU}, which corresponds to a Gaussian pdf with mean 0 and variance $\sigma^2U$, tends to a Dirac function at $V=0$. For this case it is straightforward to verify the expression for $\cG(0)$. Computing the limit for $U\to\infty$ needs more work. Note first that the solution $g(U)$ of equation \eqref{eq:lem3.1}, for large $U$, can be expressed in order of magnitude as $g(U)=\Theta(U\log U)$. This means that we can find two positive constants $a_1,a_2$ independent from $U$ such that, for large enough $U$, we have $a_1 U\log U\leq g(U)\leq a_2 U\log U$. That this is indeed possible, can be readily seen because, for sufficiently large $U$, we have $g(U)\geq0$ and $U+\kappa\geq1$, therefore we can upper and lower bound $g(U)$ from \eqref{eq:lem3.1} by observing that
$$
c_1\leq c_1+c_e\frac{g}{(U+\kappa)^2}\leq 
\max\{c_1,2c_e\sigma^2\}e^{\frac{g}{2\sigma^2(U+\kappa)}},
$$
where for the upper bound we used the inequality $e^x\geq x+1$.
These two bounds generate, immediately, the corresponding desired upper and lower bounds for $g(U)$.
A direct consequence of the order of magnitude estimate of $g(U)$ is that, since $V_1(U)=\sqrt{g^+(U)}+\mu_x\kappa$ and $V_2(U)=\sqrt{g^+(U)}-\mu_x\kappa$, we have that $V_1(U),V_2(U)$ are both $\Theta(\sqrt{U\log U})$. Using \eqref{eq:lem3.0} and \eqref{eq:lem3.mbifla1} to compute $\cG(U)$ we can see that the first term involving $c_0$ is equal to
$$
c_0\left\{\Phi\left(-\frac{V_1(U)}{\sigma\sqrt{U}}\right)+\Phi\left(-\frac{V_2(U)}{\sigma\sqrt{U}}\right)\right\}.
$$
This term tends to 0 as $U\to\infty$, since $V_i(U)/\sqrt{U}\to\infty$.
In the second term involving $c_1,c_e$, let us make the change of variables $z=\frac{V-\mu_x U}{\sigma_x\sqrt{U(U+\kappa)}}$, then we can write
\begin{multline*}
\int_{\cT(U)}\left[c_1+c_e\left(\frac{V+\mu_x\kappa}{U+\kappa}\right)^2\right]
\frac{e^{-\frac{(V-\mu_xU)^2}{2\sigma_x^2U(U+\kappa)}}}{\sqrt{2\pi\sigma_x^2U(U+\kappa)}}dV\\
=\int_{\tilde{\cT}(U)}\left[c_1+c_e\left(\mu_x+z\sigma_x\sqrt{\frac{U}{U+\kappa}}\right)^2\right]
\frac{e^{-\frac{z^2}{2}}}{\sqrt{2\pi}},
\end{multline*}
where we recall $\cT(U)=(-\infty,-V_1(U)]\cup[V_2(U),\infty)$ and we define $\tilde{\cT}(U)=(-\infty,-\tilde{V}_1(U)]\cup[\tilde{V}_2(U),\infty)$ with $\tilde{V}_1(U)=(V_1(U)+\mu_x U)/\rho(U)$, $\tilde{V}_2(U)=(V_2(U)-\mu_x U)/\rho(U)$ and $\rho(U)=\sigma_x\sqrt{U(U+\kappa)}$. Note in the last integral that the integrand is nonnegative. Furthermore integration over $\tilde{\cT}(U)$ can be regarded as integration over the whole real line after multiplying the integrand by the indicator function of the set $\tilde{\cT}(U)$. Because the indicator is nonnegative and upper bounded by 1 and $[\mu_x+z\sigma_x\sqrt{U/(U+\kappa)}]^2\leq2(\mu^2_x+z^2\sigma^2_x)$, we can upper bound the integrand by a function which does not involve $U$ and is integrable. This allows for the application of Bounded Convergence which combined with the observation that $-\tilde{V}_1(U)\to-\mu_x/\sigma_x$ and $\tilde{V}_2(U)\to-\mu_x/\sigma_x$, meaning that $\tilde{\cT}(U)$ tends to the whole real line or $\indno{\tilde{\cT}(U)}(z)\to1$, implies
\begin{multline*}
\lim_{U\to\infty}\int_{-\infty}^{\infty}\left[c_1+c_e\left(\mu_x+z\sigma_x\sqrt{\frac{U}{U+\kappa}}\right)^2\right]\indno{\tilde{\cT}(U)}(z)\frac{e^{-\frac{z^2}{2}}}{\sqrt{2\pi}}dz\\
=\int_{-\infty}^{\infty}\lim_{U\to\infty}\left[c_1+c_e\left(\mu_x+z\sigma_x\sqrt{\frac{U}{U+\kappa}}\right)^2\right]\indno{\tilde{\cT}(U)}(z)\frac{e^{-\frac{z^2}{2}}}{\sqrt{2\pi}}dz\\
=\int_{-\infty}^\infty\left[c_1+c_e\left(\mu_x+z\sigma_x\right)^2\right]\frac{e^{-\frac{z^2}{2}}}{\sqrt{2\pi}}dz=c_1+c_e(\mu_x^2+\sigma_x^2),
\end{multline*}
yielding the desired expression. This concludes the proof of our lemma.
\end{proof}

The function $\cG(U)$ introduced in Lemma\,\ref{lem:3} is very important and will simplify, considerably, the representation of the combined cost $\cC(T,\dec_T,\x_T)$. Indeed, by recalling the definition of $\x_t$ and $\Li_t$ from \eqref{eq:est} and \eqref{eq:lem2.2} respectively, we can identify the conditional expectation appearing in \eqref{eq:opt2} as $\cG(U_T)$. This means that in $\cC(T,d_T,\hat{x}_T)$ if we replace $\hat{x}_T,d_T$ with their optimum counterparts $\x_T,\dec_T$ then we have the following simple expression for the resulting combined cost
\begin{equation}
\cC(T,\dec_T,\x_T)=\cG(U_T)+c_1+c_e(\mu_x^2+\sigma_x^2).
\label{eq:mbifla}
\end{equation}
We are now in a position to reveal the optimum stopping time and finalize the desired triplet that solves the constrained optimization problem introduced in \eqref{eq:eq3}. The next theorem presents the complete solution.

\begin{theorem}
In the constraint in \eqref{eq:eq3}, let the maximal allowable cost $C$ satisfy 
$
\min\{c_0,c_1+c_e\mu_x^2\}+c_e\sigma_x^2>C>0.
$
Then, the optimum triplet $(\T,\dec_{\T},\x_{\T})$ that solves the corresponding constrained optimization problem is:
\begin{equation}
\T=\inf\{t>0:U_t\geq\gamma\},
\label{eq:th.1}
\end{equation}
where threshold $\gamma>0$ is the solution of the equation
\begin{equation}
\cG(\gamma)=C-c_1-c_e(\mu_x^2+\sigma_x^2).
\label{eq:th.2}
\end{equation}
The other two elements of the optimum triplet are given by \eqref{eq:est} for the optimum estimator and \eqref{eq:lem2.1} for the optimum detector and both, detector and estimator, need to be applied at the time of stopping $\T$.
\end{theorem}

\begin{proof}
First note that when $\min\{c_0,c_1+c_e\mu_x^2\}+c_e\sigma_x^2>C>0$, then $C-c_1-c_e(\mu_x^2+\sigma_x^2)$ takes values in the interior of the interval defined by the maximal $\cG(0)$ and minimal $\lim_{U\to\infty}\cG(U)$  value of the function $\cG(U)$. Consequently, because of the strict monotonicity and continuity of $\cG(U)$, equation \eqref{eq:th.2} has always a positive solution $\gamma=\cG^{-1}(C-c_1-c_e(\mu_x^2+\sigma_x^2))>0$ which is unique. Given that $U_0=0$; $U_t=\sum_{n=1}^t h_n^2$ is increasing; and by Assumption iii) we have $\lim_{t\to\infty}U_t=\infty$ with probability 1, we also conclude that the stopping time $\T$ defined in \eqref{eq:th.1} is almost surely finite.

Let us now show the desired optimality of the proposed triplet. Consider any alternative triplet $(T,d_T,\hat{x}_T)$ that satisfies the constraint $C\geq\cC(T,d_T,\hat{x}_T)$. Because of \eqref{eq:ineq2} and \eqref{eq:mbifla} we conclude
$$
C\geq \cC(T,d_T,\hat{x}_T)\geq\cC(T,\dec_T,\x_T)=\cG(U_T)+c_1+c_e(\mu_x^2+\sigma_x^2).
$$
The previous inequality combined with \eqref{eq:th.2} suggests that
$$
\cG(U_T)\leq C-c_1-c_e(\mu_x^2+\sigma_x^2)=\cG(\gamma)
$$
which, due to the strict decrease of $\cG(U)$, implies $U_T\geq\gamma$. From the latter we deduce that $T\geq\T$ since, by definition, $\T$ is the \textit{smallest} time instant for which this inequality holds. This establishes the optimality of the triplet $(\T,\dec_{\T},\x_{\T})$.
\end{proof}

{\sc Remark\,1:} For the completeness of our theorem we must also add that if $C\geq\min\{c_0,c_1+c_e\mu_x^2\}+c_e\sigma_x^2$ then we can verify that the optimum stopping time is $\T=0$ (no observations are needed) and the optimum joint detection/estimation structure relies, solely, on prior information. In particular if $c_0\leq c_1+c_e\mu_x^2$, we decide in favor of $\Hyp_1$ and provide as estimate the mean, that is, $\x_0=\mu_x$; whereas if $c_0> c_1+c_e\mu_x^2$, we decide in favor of $\Hyp_0$ and, of course, there is no need for any estimate.

{\sc Remark\,2:} Our theorem suggests that the optimal time to stop is when the running energy $\{U_t\}$ of the process $\{h_t\}$ exceeds the threshold $\gamma$ for the first time. This will happen with probability 1, due to \eqref{eq:wp1} in Assumption iii). This is the only requirement imposed on $\{h_t\}$ while no additional prior information is needed regarding this observed process. As far as threshold $\gamma$ is concerned, it is clear that the solution to equation \eqref{eq:th.2} can be computed numerically. 

{\sc Remark\,3:} The optimum estimate $\x_{\T}$ \textit{must} be computed when we stop at $\T$. However, initially, it is treated as an auxiliary quantity which is necessary for the application of the optimum decision rule $\dec_{\T}$. When the decision is in favor of hypothesis $\Hyp_1$, \textit{only then} $\x_{\T}$ is regarded as the actual estimate of $x$ .

{\sc Remark\,4:} As we mentioned earlier, if we select $c_e=0$ then our joint setup reduces to a pure detection problem. What is interesting in our formulation is that the optimum stopping time $\T$ is still defined through \eqref{eq:th.1} while the optimum decision function $\dec_{\T}$ becomes a likelihood ratio test where $\Li_{\T}$ is compared against the threshold $\frac{c_0}{c_1}$. This is in contrast with SPRT where, as we recall, we have a running likelihood ratio compared against two, time-varying and dependent on $\{h_t\}$, thresholds that are not possible to compute analytically. Furthermore, SPRT is optimum only when the observations are i.i.d.~whereas our simple scheme enjoys optimality even if the process $\{h_t\}$ is dependent and time varying with unknown distribution. These interesting optimality properties of our joint detection/estimation strategy are a consequence of defining the cost $\cC(T,d_T,\hat{x}_T)$ under the conditional form depicted in \eqref{eq:eq2}.

\section*{Acknowledgement}
This work was supported by the U.S. National Science Foundation under Grant CIF1064575.


\begin{thebibliography}{99}

\bibitem{Cetin13}
{\sc U. Cetin, A. Novikov and A. N. Shiryaev}, {\em Bayesian sequential estimation of a drift of fractional Brownian motion}, Seq. Anal., 32 (2013), pp. 288--296.

\bibitem{Dragalin99}
{\sc V. Dragalin, A. G. Tartakovsky and V. Veeravalli}, {\em Multihypothesis sequential probability ratio tests, Part 1: Asymptotic optimality}, IEEE Trans. Inf. Theory, 45(7) (1999), pp. 2448--2461.

\bibitem{Fellouris2013}
{\sc G. Fellouris}, {\em Asymptotically optimal parameter estimation under communication constraints}, Ann. Statist. 40(4) (2012), pp. 2239--2265. 

\bibitem{Fredriksen72}
{\sc A. Fredriksen, D. Middleton and D. Vandelinde}, {\em Simultaneous signal detection and estimation under multiple hypotheses}, IEEE Trans. Inform. Theory, 18(5) (1972), 607--614.

\bibitem{Ghosh87}
{\sc B. K. Ghosh}, {\em On the attainment of the Cramer-Rao bound in the sequential case}, Seq. Anal., 6(3) (1987), 267--288.

\bibitem{Ghosh91}
{\sc B. K. Ghosh, and P. K. Sen}, {\em Handbook of Sequential Analysis}, Marcel Dekker, New York, NY, 1991.

\bibitem{Grambsch83}
{\sc P. Grambsch}, {\em Sequential sampling based on the observed Fisher information to guarantee the accuracy of the maximum likelihood estimator}, Ann. Statist., 11(1) (1983), pp. 68--77.

\bibitem{Lerche86}
{\sc H. R. Lerche}, {\em An optimal property of the repeated significance test}, Proc. Natl. Acad. Sci. USA, 83 (1986), pp. 1546--1548.

\bibitem{Lorden77}
{\sc G. Lorden}, {\em Nearly-optimal sequential tests for finitely many parameter values}, Ann. Statist. 5 (1977), pp. 1--21.

\bibitem{Middleton68}
{\sc D. Middleton and R. Esposito}, {\em Simultaneous optimum detection and estimation of signals in noise},
IEEE Trans. Inform. Theory, 14(3) (1968), pp. 434--444.

\bibitem{Moustakides12a}
{\sc G. V. Moustakides}, {\em Optimum joint detection and estimation}, Proceedings of the IEEE International Symposium on Information Theory, Saint Petersburg, Russia (2011), pp. 2915--2919.

\bibitem{Moustakides12b}
{\sc G. V. Moustakides, G. H. Jajamovich, A. Tajer and X. Wang}, {\em Joint detection and estimation: Optimum tests and applications}, IEEE Trans. Inform. Theory, 58(7) (2012), pp. 4215--4229.

\bibitem{Pavlov90}
{\sc I. V. Pavlov}, {\em Sequential procedure of testing composite hypotheses with applications to the Kiefer-Weiss problem}, Theory Prob. Appl. 35 (1990), pp. 280--292.

\bibitem{Poor94}
{\sc H. V. Poor}, {\em An Introduction to Signal Detection and Estimation}, Springer, New York, NY, 1994.

\bibitem{Proakis08}
{\sc J. G. Proakis and M. Salehi}, {\em Digital Communications}, McGraw-Hill, New York, NY, 2008.

\bibitem{Alex98}
{\sc A. G. Tartakovsky}, {\em Asymptotic optimality of certain multihypothesis sequential tests: Non-i.i.d. case}, Statist. Infer. Stoch. Proc., 1(3) (1998), pp. 265--295.

\bibitem{Wald48}
{\sc A. Wald and J. Wolfowitz}, {\em Optimum character of the Sequential Probability Ratio Test}, Ann. Math. Statist., 19(3) (1948), pp. 326--339.

\end{thebibliography}
\end{document}